\documentclass[12pt]{article}
\input{amssym}
\newtheorem{prethm}{{\bf Theorem}}

\newenvironment{thm}{\begin{prethm}{\hspace{-0.5
em}{\bf.}}}{\end{prethm}}
\newtheorem{precor}{{\bf Corollary}}

\newenvironment{cor}{\begin{precor}{\hspace{-0.5
em}{\bf.}}}{\end{precor}}
\newtheorem{preprop}{{\bf Proposition}}

\newtheorem{preque}{{\bf Question}}

\newtheorem{preques}{{\bf Question}}

\newtheorem{prelemma}{{\bf Lemma}}

\newenvironment{lemma}{\begin{prelemma}{\hspace{-0.5
em}{\bf.}}}{\end{prelemma}}
\newtheorem{prefact}{{\bf Fact}}

\newtheorem{preobs}{{\bf Observation}}

\newtheorem{prerem}{{\bf Remark}}

\newtheorem{prefig}{{\bf Figure}}

\newtheorem{prelemm}{{\bf Lemma}}

\newtheorem{preex}{{\bf Example}}

\newtheorem{prepro}{{\bf Proposition}}

\newtheorem{prelem}{{\bf Theorem}}

\newenvironment{lem}{\begin{prelem}{\hspace{-0.5
em}{\bf.}}}{\end{prelem}}
\newtheorem{preproof}{{\bf Proof.}}

\newenvironment{proof}[1]{\begin{preproof}{\rm
               #1}\hfill{$\rule{2mm}{2mm}$}}{\end{preproof}}

\newtheorem{preconj}{{\bf Conjecture}}

\newenvironment{conj}{\begin{preconj}{\hspace{-0.5
em}{\bf.}}}
{\end{preconj}}
\newtheorem{predeff}{{\bf Definition}}

\newenvironment{deff}{\begin{predeff}{\hspace{-0.5
em}{\bf.}}}{\end{predeff}}
\def\n#1{\vbox to 3mm{\vspace{1mm}\vfill \hbox to 2.0mm{\hfill
             $#1$\hfill} \vfill }}
\def\m#1#2{\raise 0.2ex\hbox{
    ${#1_{\bf \displaystyle #2}}$}}
\def\x#1{\raise 0.5ex\hbox{
    ${#1}$}}
\usepackage{array}
\title{{\bf On the oriented perfect path double cover conjecture}}
{\small
\author{
{\sc Behrooz Bagheri Gh.} and
{\sc Behnaz Omoomi}\\
[1mm]
{\small \it  Department of Mathematical Sciences}\\
{\small \it  Isfahan University of Technology} \\
{\small \it 84156-83111,\ Isfahan, Iran}}
\date{}

\begin{document}
\maketitle
%
\begin{abstract}
An {\sf oriented perfect path double cover} ($\rm OPPDC$) of a
graph $G$ is a collection of directed paths in the symmetric
orientation $G_s$ of
 $G$ such that
each edge of $G_s$ lies in exactly one of the paths and  each
vertex  of $G$ appears just once as a beginning and just once as an
end of a path. Maxov{\'a} and  Ne{\v{s}}et{\v{r}}il (Discrete
Math. 276 (2004) 287-294)  conjectured that every graph except
two complete graphs $K_3$ and $K_5$ has an $\rm OPPDC$ and they
proved that the minimum degree of the minimal counterexample to
this conjecture   is at least four. In this paper, among some
other results,  we prove that the minimal  counterexample to this
conjecture is $2$-connected and $3$-edge-connected.
\end{abstract}

\section{\large Introduction}        
We denote by $G=(V,E)$ a finite undirected graph with no loops or
multiple edges. The {\sf symmetric orientation} of $G$, denoted
by $G_s$,   is an oriented graph obtained from $G$ by replacing
each edge of $G$ by a pair of opposite directed arcs.

A {\sf cycle double cover} ${\rm (CDC)}$ of a graph $G$ is a
collection of its cycles such that each edge of $G$ lies in
exactly two of the cycles. A well-known conjecture of
Seymour~\cite{Seymour} asserts that every simple bridgeless graph
has a ${\rm CDC}$. This problem  also motivated several related
conjectures. A {\sf small cycle double cover} ${\rm (SCDC)}$ of a
graph on $n$ vertices is a ${\rm CDC}$ with at most $n-1$ cycles.
Bondy conjectured that every simple bridgeless graph has an ${\rm
SCDC}$~\cite{Bondy}.

An {\sf oriented cycle double cover} ${\rm (OCDC)}$ of $G$ is a
collection of directed cycles in $G_S$ of length at least $3$ such that
each edge of $G_S$ lies in exactly one of the cycles.
Jaeger~\cite{Jager} conjectured that every bridgeless graph has
an oriented cycle double cover.  A {\sf small oriented  cycle
double cover} ${\rm (SOCDC)}$ of a graph $G$ on $n$ vertices is an
${\rm OCDC}$ with at most $n-1$ elements.

A {\sf perfect path double cover} ${\rm (PPDC)}$  of a graph $G$
 is a collection $\mathcal{P}$  of paths in $G$ such that each
edge  of $G$ belongs  to  exactly two  members of $\mathcal{P}$
and  each vertex  of $G$ occurs  exactly twice as an  end  of a
path in  $\mathcal{P}$~\cite{BondyPPDC}. In~\cite{PPDC} it is
proved that every simple graph has a $\rm PPDC$. The existence of
a ${\rm PPDC}$ for graphs in general is equivalent to the
existence of an ${\rm SCDC}$ for bridgeless  graph with a vertex
joined to all other vertices.
\begin{deff}~{\rm\cite{MR1825628}}
An {\sf oriented perfect path double cover  ($\rm OPPDC$)} of a
graph $G$ is a collection of directed paths in the symmetric
orientation $G_s$ such that each edge of $G_s$ lies in exactly
one of the paths and  each vertex  of $G$ appears just once as
a beginning and just once as an end of a path.
\end{deff}

 Similar to above, it can  be seen that  the existence of an ${\rm OPPDC}$ for graphs in general
is equivalent to the existence of an ${\rm SOCDC}$ for bridgeless  graphs with a
vertex joined to all other vertices.
Maxov{\'a} and
Ne{\v{s}}et{\v{r}}il in~\cite{MR1825628} showed that two complete
graphs $K_3$ and $K_5$ have no  $\rm OPPDC$, and
in~\cite{MR2046642}, they conjectured the following statement.
\begin{conj}~{\rm\cite{MR2046642}} {\rm ($\rm OPPDC$ conjecture)}
Every connected graph except $K_3$ and $K_5$ has an $\rm OPPDC$.
\end{conj}
In the following theorem a list of sufficient conditions for a
graph to admit an $\rm OPPDC$   is provided.
\begin{lem}~{\rm\cite{MR1825628}}\label{main} Let $G\ne K_3$ be a graph.
In  each of the following cases, $G$ has an $\rm OPPDC$.
\begin{description}
\item{\rm\bf(i)} $G$ is a union of two arbitrary trees.
\item{\rm\bf(ii)}  Each vertex of  $G$ has odd degree.
\item{\rm\bf(iii)} $G$ arise from a graph $G'$ which  has an $\rm OPPDC$,
by dividing one edge of $G'$.
\item{\rm\bf(iv)} $G$ has no adjacent vertices of degree greater than two.
\item{\rm\bf(v)} $G$ is a $2$-connected graph of order $n$ and $|E(G)|\le 2n-1$.
\item{\rm\bf(vi)} $G=G_1 \cup G_2$ and $V(G_1) \cap V(G_2) = \{v\}$ which $G_i$
is a graph  with an $\rm OPPDC$, for $i=1,2$.
\item{\rm\bf(vii)} $G\setminus v$ has an  $\rm OPPDC$, for some  $v\in V(G)$
of degree less than $3$.
\end{description}
\end{lem}
In~\cite{MR1825628}, Maxov{\'a} and  Ne{\v{s}}et{\v{r}}il in the following
two theorems proved that if a graph of order $n$ with a vertex $v$ of degree $3$
has no $\rm OPPDC$ then there exists a graph of order $n-1$ which has no
$\rm OPPDC$ either.
\begin{lem}\label{3}~{\rm\cite{MR1825628}}
Let $G$ be a graph, $v\in V(G)$ be a vertex of degree $3$, and
$N(v)=\{x,y,z\}$ induces $K_3$ in $G$. If $G\setminus v$ has an
$\rm OPPDC$ then $G$ has also an $\rm OPPDC$.
\end{lem}
\begin{lem}\label{3'}~{\rm\cite{MR1825628}}
Let $G$ be a graph, $v\in V(G)$ be a vertex of degree $3$, $N(v)
=\{x,y,z\}$, and $e=xz\not\in E(G)$. If $(G\setminus v)\bigcup
\{e\}$ has an $\rm OPPDC$ then $G$ also has an $\rm OPPDC$.
\end{lem}
Also in~\cite{MR1825628} the properties of a minimal counterexample to
the $\rm OPPDC$ conjecture are studied.
\begin{lem}\label{4}~{\rm\cite{MR1825628}}
If $G$ is the minimal counterexample to the $\rm OPPDC$ conjecture,
then $\delta(G)\ge 4$.
\end{lem}
The structure of this paper is as follows. In
Section~\ref{sec:MCOPPDC} the properties of the minimal
counterexample to the $\rm OPPDC$ conjecture are studied and
proved that such graphs are $2$-connected and $3$-edge-connected.
In Section~\ref{sec:SEOPPDC} some sufficient conditions for a
graph to admit an $\rm OPPDC$ are provided.

\section{\large The minimal counterexample to the $\rm OPPDC$ conjecture}\label{sec:MCOPPDC}
 In this section, among some other results, we prove that
the minimal counterexample to the $\rm OPPDC$ conjecture is
$2$-connected and $3$-edge-connected.

The complete graphs $K_3$ and $K_5$ are  only known examples of
connected graphs which have no $\rm OPPDC$. By
Theorem~\ref{main}(ii), $K_{2n}$ has an $\rm OPPDC$. It is known
that every  symmetric orientation of $K_{2n+2}$, $n\ge 3$,  has a
decomposition into $2n+1$ directed Hamiltonian
cycles~\cite{MR584161}. This decomposition forms an $\rm OPPDC$
for $K_{2n+1}$, $n\ge 3$, by deleting a fix vertex from each cycle.

By Theorem~\ref{main}(vi), if every block of graph $G$ has an $\rm
OPPDC$, then $G$ also has an $\rm OPPDC$. Let $G$ be the minimal
counterexample to the $\rm OPPDC$ conjecture. Therefore, $G$,
either is $2$-connected or at least one of its blocks is $K_3$ or
$K_5$. In the following theorem, we show that the latter can not be
happen.


For every $\rm OPPDC$ of a connected graph $G$, say $\mathcal{P}$, and every vertex
$v\in V$, let $P^v$ and $P_v$ denote the paths in $\mathcal{P}$ beginning and
 ending with $v$, respectively. Also note that we can assume in an $\rm OPPDC$
directed paths of length zero are presented only at isolated
vertices.
\begin{thm}\label{block}
Let $G=B_1\cup\ldots\cup B_k$ and $B_i$'s be blocks of $G$.
If for each $i$, $1\le i\le k$, $B_i$ has an $\rm OPPDC$ or
 $B_i=K_3$ or $K_5$, then $G$ has an $\rm OPPDC$.
\end{thm}
\begin{proof}{
Suppose $G=B_1\cup\ldots\cup B_k$ and  $B_i$'s are blocks of $G$.
If every block of $G$ has an $\rm OPPDC$, then by Theorem~\ref{main}(vi),
 $G$ also has an $\rm OPPDC$. Otherwise, at least one of $B_i$'s is $K_3$
or $K_5$. We prove the statement by induction on $k$.  \\
 If $k=2$ and $B_1=B_2=K_3$, where  $V(B_1)=\{u,v,w\}$ and $V(B_2)=\{w,x,y\}$,
then, $\mathcal{P}=\{uwxy,ywvu,xw,wuv,vwyx\}$ is an $\rm OPPDC$ of $G$.\\
If $k=2$, $B_1=K_5$ and $B_2=K_3$, where $V(B_1)=\{u,v,w,x,y\}$
and $V(B_2)=\{v,s,t\}$. Let $G'=B_1\setminus \{e=uv\}$. Then the
following is an $\rm OPPDC$ of $G'$,
$$\mathcal{P'}=\{uyxw,yvwux,wxvyu,xywv,vxuwy\}.$$
Consider four new directed paths.
 $P^t=tsvuP'^u$, $P_t=vt$, $P_s=uvs$, and $P^s=stvP'^v$.
The following is an $\rm OPPDC$ of $G$,
$$\mathcal{P}=\mathcal{P'}\cup\left\{P^t,P^s,P_t,P_s\right\}\setminus \left\{P'^u,P'^v\right\}.$$
If $k=2$ and $B_1=B_2=K_5$, where $V(B_1)=\{u,v,w,x,y\}$ and $V(B_2)=\{u',v,w',x',y'\}$.
 Then the following is an $\rm OPPDC$ of $G$,

\hspace{-6mm}
$\mathcal{P}=\{uxwyvy'w'x'u',ywxuvu'x'w'y',x'vx,u'vu,xyuwvw'u'y'x',wuyxvx'y'u'w',w'vw,\\
\hspace*{11mm}y'v,vy\}.$\\
Now let $G=G_1\cup G_2$, where $G_1=K_3$ and $G_2$ has an  $\rm
OPPDC$. Assume that $V(K_3)=\{u,v,w\}$, $v$ is a cut vertex,  and
 $\mathcal{P}_2$ is an $\rm OPPDC$ of $G_2$. Now we define four
new directed paths $P_u=P_{2v}vwu$, $P^u=uv$, $P^v=vuw$, and
$P^w=wvP^v_2$. Therefore,
$$\mathcal{P}=\mathcal{P}_2\cup \left\{P_u,P^u,P^v,P^w\right\}\setminus \left\{P^v_2,P_{2v}\right\}$$
is an $\rm OPPDC$ of $G$.\\
Finally, let  $G=G_1\cup G_2$, where  $G_1=K_5$ and $G_2$ has an
$\rm OPPDC$. Assume that $V(K_5)=\{u,v,w,x,y\}$, $v$ is a cut
vertex, and $\mathcal{P}_2$ is an $\rm OPPDC$ of $G_2$. Also let
$\mathcal{P}'$ be the $\rm OPPDC$ of $K_5\setminus \{e=uv\}$ as
given in above. Consider two new directed paths
$P_w=P_{2v}vuP'^u$ and $P^u=uvP^v_2$. Thus,
$$\mathcal{P}=\mathcal{P}'\cup \mathcal{P}_2\cup\left\{P_w,P^u\right\}\setminus \left\{P^v_2,P_{2v},P'^u\right\}$$
is an $\rm OPPDC$ of $G$.

For $k\geq 3$, the result is obtained by the induction hypothesis,
Theorem~\ref{main}(vi) and the above argument.
 }\end{proof}
Theorem~\ref{block} implies that the minimal counterexample
 to the $\rm OPPDC$ conjecture has no cut vertex.
\begin{cor}\label{k=2}
The minimal counterexample to the $\rm OPPDC$ conjecture is $2$-connected.
\end{cor}
Moreover, Theorem~\ref{block} directly concludes the following corollaries.
A {\sf block graph} is a graph for which each block is a clique.
\begin{cor}
Every block graph has an $\rm OPPDC$.
\end{cor}
Since the line graph of every tree is a block graph,
we have the following corollary.
\begin{cor}\label{L(T)}
For every tree $T$, $L(T)$  has an $\rm OPPDC$.
\end{cor}
For line graphs the following  result is also obtained from  Theorem~\ref{main}(ii).
\begin{cor}
If the degree of no adjacent vertices in $G$ have the  same
parity, then the line graph $L(G)$ has an $\rm OPPDC$.
\end{cor}
The following lemmas are necessary to prove our next theorem.
\begin{lemma}\label{K5=K5}
If  $G_1=G_2=K_5$ and $G=G_1\cup G_2\cup \{ uu', vv'\}$, where
$\{u,v\}\in V(G_1)$ and $\{u',v'\}\in V(G_2)$, then $G$ has an $\rm OPPDC$.
\end{lemma}
\begin{proof}{
Let $G_1=G_2=K_5$,  $V(G_1)=\{u,v,w,x,y\}$, and
$V(G_2)=\{u',v',w',x',y'\}$.
 Therefore the following is an $\rm OPPDC$ of $G=G_1\cup G_2\cup \{ uu', vv'\}$.\\
$\mathcal{P}=\{uxywvv'y'u'x'w',xvwu,wxuvy,yuu',vuwyx,v'x'y'w'u'uyvxw,x'u'w'v',w'x'v'u'y',\\
\hspace*{11mm}y'v'v,u'v'w'y'x'\}.$
}\end{proof}
\begin{lemma}\label{G=K5}
Let $G_1=K_5$ and $G_2$ be a graph with an $\rm OPPDC$.
If $G=G_1\cup G_2\cup \{ uu', vv'\}$, where $\{u,v\}\in V(G_1)$
and $\{u',v'\}\in V(G_2)$, then $G$ has an $\rm OPPDC$.
\end{lemma}
\begin{proof}{
Let $V(G_1)=\{u,v,w,x,y\}$, $\mathcal{P}'$ be the $\rm OPPDC$ of $G_1\setminus \{e=uv\}$
given in the proof of Theorem~\ref{block}, and $\mathcal{P}_2$ be an $\rm OPPDC$ of $G_2$.
Now set four new directed paths. $P_v=P_{2u'}u'uv$, $P^u=uu'$, $P_{v'}=P'_vvv'$,
and $P_w=P_{2v'}v'vuP^{'u}$. Thus,
$$\mathcal{P}=\mathcal{P}'\cup \mathcal{P}_2\cup\left\{P^u,P_v,P_w,P_{v'}\right\}\setminus \left\{P'^u,P'_v,P_{2u'},P_{2v'}\right\}$$
is an $\rm OPPDC$ of $G$.
}\end{proof}
By Corollary~\ref{k=2}, the minimal counterexample to the $\rm OPPDC$ conjecture
is bridgeless, therefore if $G$ has an edge cut $F$ of size $2$,
then the edges of $F$ are vertex disjoint. In the next theorem,
we show that $G$ has no vertex disjoint edge cut of size $2$.
\begin{thm}\label{K'=3}
The minimal counterexample to the $\rm OPPDC$ conjecture is
$3$-edge-connected.
\end{thm}
\begin{proof}{
Let $G$ be the minimal counterexample to the $\rm OPPDC$
conjecture. Suppose, on the contrary,  that $G$ has an edge cut of
size $2$,  say $F$. By Corollary~\ref{k=2}, $F$ is vertex
disjoint.
 Let $F=\{uv,wx\}$, and $G_1$ and $G_2$ be the components of
 $G\setminus F$ such that $u,w\in V(G_1)$.

 If $G_1$ and $G_2$
have no $\rm OPPDC$, then by minimality of $G$ and by Theorem~\ref{4},
 $G_1$ and $G_2$ are isomorphic to $K_5$. Therefore by Lemma~\ref{K5=K5},
$G$ has an $\rm OPPDC$ which is a contradiction. Now without loss
of generality, suppose that only $G_1$ has an $\rm OPPDC$. By
minimality of $G$ and  Theorem~\ref{4}, $G_2$ is isomorphic to
$K_5$; thus by Lemma~\ref{G=K5}, $G$ has an $\rm OPPDC$ which is
a contradiction.

It remains to consider the case that,
  $G_1$ and $G_2$ have an $\rm OPPDC$. Assume
that $\mathcal{P}_i$ is an $\rm OPPDC$ of $G_i$, $i=1,2$.
Now we define four new directed paths $P=P_{1u}uvP^v_2$, $P^v=vu$,
$P'=P_{1w}wxP^x_2$, and $P^x=xw$. Therefore,
$$\mathcal{P}=\mathcal{P}_1\cup \mathcal{P}_2\cup\left\{P,P',P^v,P^x\right\}\setminus \left\{P_{1u},P_{1w},P^v_2,P^x_2\right\}$$
is an $\rm OPPDC$ of $G$. This contradiction implies that $G$
is $3$-edge-connected.
}\end{proof}
\section{\large Some sufficient conditions for existence of an $\rm OPPDC$ }\label{sec:SEOPPDC}
In this section, we prove some sufficient conditions for a graph
to admit an $\rm OPPDC$.
Since the minimal counterexample to the $\rm OPPDC$ conjecture is
$2$-connected, first we consider the $\rm OPPDC$ conjecture for
$2$-connected graphs.

An {\sf ear-decomposition} of a $2$-connected graph $G$ is a
decomposition of $E(G)$ to subgraphs
$G_0=C_0\subset G_1\subset \ldots \subset G_k=G$ such that
$C_0$ is a cycle and for $i$, $2\le i\le k$,
$G_i\setminus G_{i-1}$ is a simple path in $G_i$, with
only two end vertices in $G_{i-1}$.
\begin{thm}\label{ear}
If a $2$-connected graph $G$ has an ear-decomposition
$G_0=C_0\subset G_1\subset \ldots \subset G_k=G$ such that
$G_i\setminus G_{i-1}=P_i$ is a path of length at least $2$, for
$i=1,\ldots,k$, and $C_0\ne K_3$, then $G$ has an $\rm OPPDC$.
\end{thm}
\begin{proof}{
We prove the statement by induction on $k$.
For $k=0$, $G$ is a cycle and
 the following is an $\rm OPPDC$ of cycle  $C=[v_1,v_2,\ldots,v_n]$.
$$\mathcal{P}=\{v_nv_{n-1}, v_{n-1}v_{n-2}\ldots v_2v_1v_n, v_{n-2}v_{n-1}v_nv_1\}\cup\left(\cup_{i=1}^{n-3}\{v_iv_{i+1}\}\right).$$
 Now by induction on $k$ and by
Theorem~\ref{main}(vii) and (iii), an $\rm OPPDC$ of $G$ is
obtained.
\vspace{-1mm}
}\end{proof}
The following corollary provides a condition for every ear decomposition
 of the  minimal counterexample to the $\rm OPPDC$ conjecture.
\begin{cor}
Every ear-decomposition of the minimal counterexample to the
$\rm OPPDC$ conjecture has at least one ear of length $1$.
\end{cor}
\begin{thm}
Let $G$ be a connected graph. If $E(G)$ is partitioned to a cycle
$C$ of length at least $4$ and a connected graph $G'$ such that
$G'$ has an $\rm OPPDC$ and $|V(C)\setminus V(G')|\ge 2$, then $G$
also has an $\rm OPPDC$.
\end{thm}
\begin{proof}{
If $|V(C)\cap V(G')|=1$, then by Theorem~\ref{main}(vi), $G$ has
an $\rm OPPDC$. Now, suppose that $|V(C)\cap V(G')|\ge 2$. Let
$\mathcal{P'}$ be an $\rm OPPDC$ of $G'$ and
$C=[v_1,v_2,\ldots,v_k]$.

 If there exist two
vertices $v_i$ and $v_j$ , $i<j$, in $V(C)\setminus V(G')$ and
two vertices $v_r$ and $v_s$ in $V(C)\cap V(G')$  both are in the
same segment of $C$ divided by $v_i$ and $v_j$, then without loss
of generality, we can assume that
 $1\le i<j<r<s\le k$. Thus, we can  find an $\rm OPPDC$ for $G$ as
follows. Let $P^{v_i}=v_iv_{i-1}v_{i-2}\ldots v_sP'^{v_s},
P^{v_s}=v_sv_{s-1}\ldots v_i,\
 P^{v_j}=v_jv_{j+1}\ldots v_rP'^{v_r},$ and $P^{v_r}=v_rv_{r+1}\ldots v_j.$
Now, let $\mathcal{\widetilde{P}}^{v_i}$ and
$\mathcal{\widetilde{P}}^{v_s}$ be the collections of directed
paths obtained by separating the paths $P^{v_i}$ and $P^{v_s}$ on
the vertices of $V(C)\setminus (V(G')\cup\{v_j\})$. Thus,  the following is an
$\rm OPPDC$ of $G$,
$$\mathcal{P}=\mathcal{P'}\cup \mathcal{\widetilde{P}}^{v_i}\cup \mathcal{\widetilde{P}}^{v_s}
\cup\{P^{v_j},P^{v_r}\}\setminus \{P'^{v_r},P'^{v_s}\}.$$
 Otherwise,  $C=[v_1,v_2,v_3,v_4]$
and $V(C)\cap V(G')=\{v_1,v_3\}$. In this case, we define four new directed paths
 $P_{v_2}=v_1v_4v_3v_2,\ P^{v_2}=v_2v_1P'^{v_1},\ P^{v_4}=v_4v_1v_2v_3,$
and $P_{v_4}=P'_{v_3}v_3v_4.$ Now, the following is an $\rm OPPDC$ of $G$.
$$\mathcal{P}=\mathcal{P'}\cup \{P_{v_2},P^{v_2},P^{v_4},P_{v_4}\}\setminus \{P'^{v_1},P'_{v_3}\}.$$
}\end{proof}

\begin{cor}\label{CD}
Let $G$ be a connected graph. If $E(G)$ is partitioned to
a collection of cycles $\{C_1,C_2,\ldots,C_k\}$ such that
for each $i$, $2\le i\le k$, $|V(C_i)\setminus\cup_{j<i}V(C_j)|\ge 2$
and $C_1\ne K_3$, then $G$ has an $\rm OPPDC$.
\end{cor}
In the following theorem, we give a sufficient condition
for the existence of an $\rm OPPDC$ in
graphs of minimum degree at most three.

\begin{thm}\label{thm:3,4}
If $G\ne K_3$ is a graph with $\Delta(G)\le 4$ and
$\delta(G)\le 3$, then $G$ has an $\rm OPPDC$.
\end{thm}
\begin{proof}{We proceed by induction on the order of graph, $n$.
For $n=2$ the statement is trivial. For $n\ge 3$, suppose
$\deg(v)=\delta(G)\le 3$. If $d(v)=1$ or $2$, then
$G'=G\setminus v$ is a graph of order
$n-1$, $\Delta(G')\le 4$, and $\delta(G')\le 3$. Therefore, by
the induction hypothesis $G'$ has an $\rm OPPDC$, and  by
Theorem~\ref{main}(vii), $G$  also has an $\rm OPPDC$.\\
Let $\deg(v)=3$ and $N(v)=\{x,y,z\}$. Now, if $N(v)$
induces $K_3$, then by the induction
hypothesis and by Theorem~\ref{3}, $G$ has an $\rm OPPDC$.
Otherwise, let $e=xz\notin E(G)$. Thus by the induction hypothesis,
 $G\setminus v \cup\{e\}$ has an $\rm OPPDC$.
Therefore by Theorem~\ref{3'}, $G$ admits an $\rm OPPDC$.
}\end{proof}
\begin{cor}
Every separable $4$-regular graph  has an $\rm OPPDC$.
\end{cor}
\begin{proof}{
If $G$ is a separable $4$-regular graph, then
every block, $G'$, of $G$ is a graph with $\Delta(G')\le 4$ and $\delta(G')\le 3$.
 Therefore,  by Theorems~\ref{thm:3,4} and   ~\ref{main}(vi), $G$ has an $\rm OPPDC$.
}\end{proof}
%
Following theorem guarantees the exsitence of an $\rm OPPDC$ for the large family of graphs.
The {\sf Cartesian product}, $G\square H$  of two graphs $G$ and $H$
is the graph with vertex set $V(G)\times V(H)$ and two vertices $(u,v)$
and $(x,y)$ are  adjacent if and only if either $u=x$  and  $vy\in E(H)$,
  or  $ux\in E(G)$  and $v=y$.
In the following theorem we prove that the existence of an $\rm
OPPDC$ for two graphs $G$ and $H$, provides an $\rm OPPDC$ for the
Cartesian product of $G$ and $H$.
\begin{thm}\label{GsquareH}
If  $G$ and $H$ have an $\rm OPPDC$, then $G\square H$ also has an
$\rm OPPDC$.
\end{thm}
\begin{proof}{Suppose that $\mathcal{P}$ and $\mathcal{Q}$ are the
$\rm OPPDC$ of $G$ and $H$, respectively. Let $\mathcal{R}=\{P_uQ^v\ : \ (u,v)\in V(G\square
H)\}$, where $P_u\in \mathcal{P}$ is the directed path ends with $u$ in the
copy of $G$ in  $G\square H$  corresponding to the vertex $v$ in $H$, and  $Q^v\in \mathcal{Q}$ is the directed path
 strats from $v$ in the copy of $H$ in $G\square H$ corresponding to the vertex $u$ in $G$. It can be seen that every directed edge of the  symmetric orientation of $G\square H$ is covered by one path in $\mathcal{R}$ and
every vertex $(u,v)$ appears just once as a beginning
and once as an end of a path in $\mathcal{R}$. Therefore, $\mathcal{R}$
 is an $\rm OPPDC$ of $G\square H$.
}\end{proof}
\begin{cor}
If  $G$  has an $\rm OPPDC$, then $G^l=\overbrace{G\square \cdots\square G}^{l\ times}$
also has an $\rm OPPDC$.
\end{cor}

Theorem~\ref{GsquareH} concludes that the $\rm OPPDC$ conjecture
holds for some well known families of graphs, such as Cartesian
product of cycles, paths, wheels, complete graphs, and  complete
bipartite graphs. In the following an   $\rm OPPDC$ for the
complete bipartite graph is given.

\begin{thm}\label{Kn,m}
Every $K_{n,m}$  has an $\rm OPPDC$.
\end{thm}
\begin{proof}{
Let $V(K_{n,m})=\{v_1,\ldots,v_n;w_1,\ldots,w_m\}$  and $E(K_{n,m})=\{v_iw_j\ : \ 1\le i\le n,\ 1\le j\le m\}$.
We proceed by induction on  $m$.
Define $P^{v_1}_{n,1}= v_1w_1,\ P^{w_1}_{n,1}=w_1v_n$, and $P^{v_i}_{n,1}=v_iw_1v_{i-1}$, for $2\le i\le n$. Therefore,
$$\mathcal{P}_{n,1}=\{P^{w_1}_{n,1},P^{v_i}_{n,1}\ : \ 1\le i\le n\}$$ is an $\rm OPPDC$ of $K_{n,1}$.\\
Now for $m\ge 2,$  define $P^{v_1}_{n,m}= v_1w_m,\ P^{w_m}_{n,m}=w_mv_nP^{v_n}_{n,m-1},\ P^{v_i}_{n,m}=v_iw_mv_{i-1}P^{v_{i-1}}_{n,m-1}$,
 for $2\le i\le n$, and $P^{w_j}_{n,m}=P^{w_j}_{n,m-1}$, for $2\le j\le m-1$. Thus,
$$\mathcal{P}_{n,m}=\{P^{v_i}_{n,m},P^{w_j}_{n,m}\ : \ 1\le i\le n,\ 1\le j\le m\},$$
is an  $\rm OPPDC$ of $K_{n,m}$.
}\end{proof}

\bibliographystyle{plain}

\end{document}